\documentclass[12pt]{amsart}

\newif\ifdraft\draftfalse

\usepackage{amssymb}
\usepackage{amsthm}
\usepackage{eucal}
\usepackage[english]{babel}
\usepackage{nicefrac}

\makeatletter
\def\@begintheorem#1#2[#3]{%
    \def\naam{#1}
  \deferred@thm@head{\the\thm@headfont \thm@indent
    \@ifempty{#1}{\let\thmname\@gobble}{\let\thmname\@iden}%
    \@ifempty{#2}{\let\thmnumber\@gobble}{\let\thmnumber\@iden}%
    \@ifempty{#3}{\let\thmnote\@gobble}{\let\thmnote\@iden}%
    \thm@swap\swappedhead\thmhead{#1}{#2}{#3}%
    \the\thm@headpunct
    \thmheadnl 
    \hskip\thm@headsep
  }%
  \ignorespaces}
\makeatother

\newcommand{\kantlijndraft}[1]{\ifdraft\hspace{-{-1}stskip}%
\vadjust{\vspace{-1mm}\smash{\llap{{\tt #1}\hspace{8mm}}}\vspace{1mm}}\fi}

\def\voegToe#1#2#3{\immediate\write1{\string\newlabel{#1}{{#2}{#3}}}}

\newcommand{\thlabel}[1]{\voegToe{#1}{\naam\noexpand~\thetheorem}{\thepage}\kantlijndraft{#1}}

\makeatletter
\renewcommand{\label}[1]{\voegToe{#1}{\@currentlabel}{\thepage}\kantlijndraft{#1}}
\makeatother

\newtheorem{theorem}{Theorem}[section]
\newtheorem{lemma}[theorem]{Lemma}
\newtheorem{corollary}[theorem]{Corollary}
\newtheorem{question}[theorem]{Question}

\theoremstyle{definition}
\newtheorem{example}[theorem]{Example}

\theoremstyle{remark}

\numberwithin{equation}{section}

\newtheorem{claim2}{\sc Claim}



\newcommand{\pichar}[1]{\ensuremath{\pi\chi(#1)}}

\newcommand{\sse}{\subseteq}						
\newcommand{\minus}{\backslash}						
\newcommand{\Un}{\bigcup}							
\newcommand{\un}{\cup}								
\newcommand{\Meet}{\bigcap}							
\newcommand{\meet}{\cap}							
\newcommand{\es}{\varnothing}						

\newcommand{\scr}[1]{\ensuremath{\mathcal{#1}}}

\newcommand{{\la}}{\lambda}

\def\cprime{$'$}

\def\sapirovskii{{\v{S}}apirovski{\u\i}}
\def\arhangelskii{Arhangel{\cprime}ski{\u\i}}

\begin{document}

\title{On cardinality bounds involving the weak Lindel\"of degree}

\author{A. Bella}\address{Department of Mathematics, University of Catania, Viale A. Doria 6, 95125 Catania, Italy}
\email{bella@dmi.unict.it}
\author{N. Carlson}\address{Department of Mathematics, California Lutheran University, 60 W. Olsen Rd, MC 3750, Thousand Oaks, CA 91360 USA}
\email{ncarlson@callutheran.edu}



\thanks{The research that led to the present paper was partially
supported by
a grant of the
group GNSAGA of INdAM}

\begin{abstract} 
We give a general closing-off argument in Theorem~\ref{thm1} from which several corollaries follow, including (1) if $X$ is a locally compact Hausdorff space then $|X|\leq 2^{wL(X)\psi(X)}$, and (2) if $X$ is a locally compact power homogeneous Hausdorff space then $|X|\leq 2^{wL(X)t(X)}$. The first extends the well-known cardinality bound $2^{\psi(X)}$ for a compactum $X$ in a new direction. As $|X|\leq 2^{wL(X)\chi(X)}$ for a normal space $X$~\cite{BGW1978}, this enlarges the class of known Tychonoff spaces for which this bound holds. In~\ref{thm2} we give a short, direct proof of (1) that does not use~\ref{thm1}. Yet~\ref{thm1} is broad enough to establish results much more general than (1), such as  if $X$ is a regular space with a $\pi$-base $\scr{B}$ such that $|B|\leq 2^{wL(X)\chi(X)}$ for all $B\in\scr{B}$, then $|X|\leq 2^{wL(X)\chi(X)}$. 

Separately, it is shown that if $X$ is a regular space with a $\pi$-base whose elements have compact closure, then $|X|\leq 2^{wL(X)\psi(X)t(X)}$. This partially answers a question from \cite{BGW1978} and gives a third, separate proof of (1). We also show that if $X$ is a weakly Lindel\"of, normal, sequential space with $\chi(X)\leq 2^{\aleph_0}$, then $|X|\leq 2^{\aleph_0}$.

Result (2) above is a new generalization of the cardinality bound $2^{t(X)}$ for a power homogeneous compactum $X$ (Arhangel'skii, van Mill, and Ridderbos \cite{avr2007}, De la Vega in the homogeneous case~\cite{dlv2006}). To this end we show that if $U\sse clD\sse X$, where $X$ is power homogeneous and $U$ is open, then $|U|\leq |D|^{\pi_{\chi}(X)}$. This is a strengthening of a result of Ridderbos~\cite{rid2006}. 

\end{abstract}

\maketitle


\section{Introduction.} 
The \emph{weak Lindel\"of degree} $wL(X)$ of a space $X$ is the least cardinal $\kappa$ such that for every open cover $\scr{U}$ of $X$ there exists $\scr{V}\in[\scr{U}]^{\leq\kappa}$ such that $X=cl(\Un\scr{V})$. In 1978 Bell, Ginsburg, and Woods~\cite{BGW1978} gave an example of a Hausdorff space $X$ for which $|X| > 2^{wL(X)\chi(X)}$ and showed that if $X$ is normal then $|X|\leq 2^{wL(X)\chi(X)}$.  It was asked in Question 1 in~\cite{Hodel2006} whether the normality condition can be weakened to regular, yet to the authors' knowledge the question is even open under the Tychonoff assumption (see Question~\ref{q1} below). Dow and Porter~\cite{DowPorter82} showed that $|X|\leq 2^{\psi_c(X)}$ if $X$ is H-closed, giving another class of spaces for which $|X|\leq 2^{wL(X)\chi(X)}$. As  $|X|\leq 2^{L(X)\chi(X)}$ \cite{arh1969} and $|X|\leq 2^{c(X)\chi(X)}$ \cite{HJ} for every Hausdorff space, it also follows trivially that if $X$ is Lindel\"of or $X$ has countable chain condition then $|X|\leq 2^{wL(X)\chi(X)}$. More recently, Gotchev~\cite{gotchev} has shown that if $X$ has a regular $G_\delta$-diagonal then the cardinality of $X$ has the even stronger bound $wL(X)^{\chi(X)}$. 

We then have five classes of Hausdorff spaces $X$ for which $|X|\leq 2^{wL(X)\chi(X)}$: Lindel\"of, c.c.c., normal, H-closed, and spaces having a regular $G_\delta$-diagonal. All of these properties can be thought of as ``strong'' in some sense. One may expect that these properties would be strong, as the cardinal invariant $wL(X)$ is generally regarded as a substantial weakening of the Lindel\"of degree $L(X)$, to such an extent that even if $X$ is normal it is not guaranteed that $wL(X)=L(X)$. (However, it is straightforward to see that if $X$ is paracompact then $wL(X)=L(X)$, as pointed out in 4.3 in~\cite{BGW1978}). Also, $wL(X)$ is small enough of an invariant so that $wL(X)\leq c(X)$, where $c(X)$ is the cellularity of $X$, itself regarded as being ``small''. 

In $\S2$ of this note we give other classes of spaces $X$ for which $|X|\leq 2^{wL(X)\chi(X)}$. It is shown in Corollary~\ref{cor4} that if $X$ is either Urysohn or quasiregular and has a dense subset $D$ such that each $d\in D$ has a closed neighborhood that is $H$-closed, normal, Lindel\"of, or c.c.c., then $|X|\leq 2^{wL(X)\chi(X)}$. The fundamental technique we develop is a closing-off argument given in the proof of Theorem~\ref{thm1}. This theorem is a modified version of Theorem 2.5(b) in~\cite{Carlson2013}. The latter theorem assumes a space with a dense set of isolated points, which we generalize in Theorem~\ref{thm1} to spaces with an open $\pi$-base $\scr{B}$ such that $|B|$ is still ``small'' for each $B\in\scr{B}$; that is, $|B|\leq 2^{\kappa}$ for a suitably defined cardinal $\kappa$. 

As every locally compact space is of pointwise countable type and thus $\chi(X)=\psi(X)$, it follows from Corollary~\ref{cor4} that the cardinality of a locally compact space is at most $2^{wL(X)\psi(X)}$ (Corollary~\ref{cor6}). This result then generalizes the well-known bound $2^{\psi(X)}$ for a compactum $X$ in a new direction. It is clear that $2^{\psi(X)}$ is not itself a cardinality bound for all locally compact spaces, as witnessed by a discrete space of size $2^\mathfrak{c}$. 

In Theorem~\ref{thm2} we give a short, direct proof that $2^{wL(X)\chi(X)}$ is a cardinality bound for any Hausdorff space $X$ that is locally H-closed, regular and locally normal, locally Lindel\"of, or locally c.c.c. (Note, however, that this also follows from Corollary~\ref{cor4} in the case where $X$ is Urysohn or quasiregular). It follows (again) that locally compact spaces $X$ satisfy $|X|\leq 2^{wL(X)\psi(X)}$. The result that every regular, locally normal space $X$ satisfies $2^{wL(X)\chi(X)}$ is a presents a proper improvement of the Bell, Ginsburg, and Woods result for normal spaces. 

In Theorem~\ref{thm1.5} we show that if $X$ is a regular space with a $\pi$-base whose elements have compact closure, then $|X|\leq 2^{wL(X)\psi(X)t(X)}$. The proof does not use the main Theorem~\ref{thm1} and gives a third proof that the cardinality of a locally compact space $X$ satisfies $|X|\leq 2^{wL(X)\psi(X)}$. It also give a partial answer to Question 4.1 in \cite{BGW1978}.

Also in $\S2$ we show that if $X$ is a weakly Lindel\"of, normal, sequential space satisfying $\chi(X)\leq 2^{\aleph_0}$, then $|X|\leq 2^{\aleph_0}$. One might think of this as an analogue of \arhangelskii's result that if $X$ is a Lindel\"of, Hausdorff, sequential space with $\psi(X)\leq 2^{\aleph_0}$ then $|X|\leq 2^{\aleph_0}$.

The main closing-off argument given in Theorem~\ref{thm1} also has implications for the cardinality of spaces with homogeneity-like properties, which we give in $\S3$. Recall a space $X$ is \emph{homogeneous} if for all $x,y\in X$ there exists a homeomorphism $h:X\to X$ such that $h(x)=y$, and $X$ is \emph{power homogeneous} if there exists a cardinal $\kappa$ such that $X^\kappa$ is homogeneous. It is shown in Corollary~\ref{lcpowerhomog} that if $X$ is locally compact and power homogeneous then $|X|\leq 2^{wL(X)t(X)}$. This represents an extension of the cardinality bound $2^{t(X)}$ for a power homogeneous compactum $X$ given by~\arhangelskii, van Mill, and Ridderbos~\cite{avr2007} in a new direction. (De la Vega first established the bound $2^{t(X)}$ for compact homogeneous spaces in~\cite{dlv2006}).  A key ingredient is Lemma~\ref{phlemma2}, which slightly improves Ridderbos' bound $d(X)^{\pichar{X}}$~\cite{rid2006} for the cardinality of a power homogeneous space $X$ and uses modified techniques from that paper. The cardinality bound $2^{wL(X)t(X)}$ for locally compact, power homogeneous spaces is then a ``companion bound'' to the bound $2^{wL(X)\psi(X)}$ for general locally compact spaces, both fundamentally proved with the same closing-off argument given in Theorem~\ref{thm1}. Such companion bounds for power homogeneous spaces also appear in~\cite{Carlson2013},~\cite{CPR2012}, and~\cite{CR2012}.

For definitions of cardinal invariants and other notions not defined in this note, we refer the reader to~\cite{Engelking},~\cite{Juhasz}, and~\cite{por88}. \emph{All spaces are assumed to be Hausdorff.}

\section{A closing-off argument involving the weak-Lindel\"of degree.}

In this section we aim towards demonstrating that minor separation requirements on a space $X$ (Urysohn or quasiregular) and the existence of a $\pi$-base $\scr{B}$ for $X$ such that $|B|\leq 2^{wL(X)\chi(X)}$ for all $B\in\scr{B}$ are sufficient to guarantee that $|X|\leq 2^{wL(X)\chi(X)}$ (Corollary~\ref{cor2}).

In Theorem~\ref{thm1} below we give the main closing-off argument at the core of our cardinality bound results. This theorem is similar to Theorem 2.5(b) in~\cite{Carlson2013}, but rather than the requirement that $X$ has a dense set of isolated points, we make the weaker assumption that $X$ has a $\pi$-base with ``small'' elements. 

Recall that for a space $X$, the $\theta$-\emph{closure} of a subset $A\sse X$ is $cl_\theta A=\{x\in X: clU\meet A\neq\es\textup{ for all open sets }U\textup{ containing }x\}$. A subset $D\sse X$ is $\theta$-\emph{dense} if $cl_\theta D=X$. The $\theta$-\emph{density} of space $X$ is $d_\theta(X)=\min\{|D|:D\textup{ is }\theta\textup{-dense in }X\}$. It is straightforward to see that if $X$ is regular then $cl_\theta A=clA$ for all $A\sse X$, and if $X$ is \emph{quasiregular}; that is, every non-empty open set contains a non-empty regular-closed set, then $d_\theta(X)=d(X)$.

\begin{theorem}\label{thm1}
Let $X$ be a space and $\kappa$ a cardinal such that $wL(X)t(X)\leq\kappa$. Suppose $X$ has an open $\pi$-base $\scr{B}$ such that $|B|\leq 2^\kappa$ for all $B\in\scr{B}$. Let $\scr{C}$ be a cover of $X$ consisting of compact subsets $C$ of $X$ such that $\chi(C,X)\leq\kappa$. Then there exists a subcollection $\scr{C}^\prime\sse\scr{C}$ such that $X=cl_\theta\left(\Un\scr{C}^\prime\right)$ and $|\scr{C}^\prime|\leq 2^\kappa$.
\end{theorem}

\begin{proof}
For every $C\in\mathcal{C}$, we fix a collection $\mathcal{U}_C$ of open subsets of $X$ that forms a neighborhood base at $C$ such that $|\mathcal{U}_C|\leq\kappa$. If $\mathcal{C}^\prime\subseteq\mathcal{C}$, then define $\mathcal{U}(\mathcal{C}^\prime) = \bigcup\{ \mathcal{U}_C : C\in\mathcal{C}^\prime\}$. We note that each $C\in\mathcal{C}$ is a $G^c_\kappa$-set, as defined in Definition 3.3 in \cite{CPR2012}. $G^c_\kappa$-sets are also referred to as regular $G_\kappa$-sets.

By induction we build an increasing sequence $\{A_\alpha : \alpha <\kappa^+\}$ of open subsets of $X$ and an increasing chain $\{\mathcal{C}_{\alpha} : \alpha < \kappa^+\}$ of subsets of $\mathcal{C}$ such that
\begin{enumerate}
\item $|\mathcal{C}_\alpha|\leq 2^\kappa$ and $|A_\alpha|\leq 2^\kappa$.
\item $\mathcal{C}_\alpha$ covers $clA_\alpha$,
\item if $\mathcal{V}\in[\mathcal{U}(\mathcal{C}_\alpha)]^{\leq\kappa}$ is such that $X\setminus cl\left(\Un\scr{V}\right)\neq\es$, then $A_{\alpha+1}\setminus cl\left(\Un\scr{V}\right)\neq\es$.
\end{enumerate}

For limit ordinals $\beta<\kappa^+$, we let $A_\beta=\bigcup_{\alpha<\beta} A_\alpha$. Then $|A_\beta|\leq 2^\kappa$ and hence $d(clA_\beta)\leq 2^\kappa$. By Lemma 3.5 in \cite{CPR2012} we obtain a collection $\mathcal{C}_\beta$ with the properties needed in (1) and (2). 

Consider a successor ordinal $\beta+1$. As $\scr{B}$ is a $\pi$-base, for each $\mathcal{V}\in[\mathcal{U}(\mathcal{C}_\beta)]^{\leq\kappa}$ for which $X\setminus cl\left(\Un\scr{V}\right)\neq\es$, there exists $B_{\scr{V}}\in\scr{B}$ such that $B_{\scr{V}}\sse X\setminus cl\left(\Un\scr{V}\right)\neq\es$ and $|B_{\scr{V}}|\leq 2^\kappa$. Define 
$$A_{\beta+1}=A_\beta\un\bigcup\left\{B_{\scr{V}}:\mathcal{V}\in[\mathcal{U}(\mathcal{C}_\beta]^{\leq\kappa}\textup{ and }X\setminus cl\left(\Un\scr{V}\right)\neq\es.\right\}$$

Observe that $A_\alpha$ is open for all $\alpha<\kappa^+$. As $|A_\beta|\leq 2^\kappa$, $|[\mathcal{U}(\mathcal{C}_\beta)]^{\leq\kappa}|\leq 2^\kappa$, and each $|B_{\scr{V}}|\leq 2^\kappa$, we have $|A_{\beta+1}|\leq 2^\kappa$ and thus $d(clA_{\beta+1})\leq 2^\kappa$. We again use Lemma 3.5 in \cite{CPR2012} to obtain $\mathcal{C}_{\beta+1}$.

Let $\mathcal{C}^\prime = \bigcup_{\alpha<\kappa^+}\mathcal{C}_\alpha$ and $F=\bigcup_{\alpha<\kappa^+} clA_\alpha$. Note that $|\scr{C}^\prime|\leq 2^\kappa$,  $F$ is closed since $t(X)\leq\kappa$, $F=cl\left(\Un_{\alpha<\kappa^+}A_\alpha\right)$, and $\mathcal{C}^\prime$ covers $F$. We now show that $X\sse cl_\theta\left(\Un\scr{C}^\prime\right)$. Suppose not and pick a point $x\in X\minus cl_\theta\left(\bigcup\mathcal{C}^\prime\right)$. Then there exists an open set $W$ containing $x$ such that $\bigcup\mathcal{C}^\prime\sse X\minus clW$. Hence for each $C\in\scr{C}^\prime$, we see that $C\sse X\minus clW$ and since $\mathcal{U}_C$ forms a neighborhood base at $C$ there exists $U_C\in\mathcal{U}_C$ such that $C\sse U_C\sse X\minus clW$. Then, as $\scr{C}^\prime$ covers $F$, $\mathcal{U} = \{ U_C : C\in\mathcal{C}^\prime\}$ is an open cover of $F$. As $wL(X)\leq \kappa$ and $F$ is regular-closed, it follows that $wL(F,X)\leq\kappa$ as $wL(X)$ is hereditary on regular-closed sets.We may find $\mathcal{V}\in[\mathcal{U}]^{\leq\kappa}$ such that $F\subseteq cl\left(\Un\scr{V}\right)$. Since $W\meet \bigcup\{V:V\in\mathcal{V}\}=\es$, we see that $x\in X\setminus cl\left(\Un\scr{V}\right)$. As $\mathcal{V}\in[\mathcal{U}(\mathcal{C}_\alpha)]^{\leq\kappa}$ for some $\alpha<\kappa^+$, by condition (3) we have $A_{\alpha+1}\setminus cl\left(\Un\scr{V}\right)$. But this is a contradiction since $A_{\alpha+1}\sse F\subseteq cl\left(\Un\scr{V}\right)$. Thus, $X\sse cl_\theta\left(\Un\scr{C}^\prime\right)$.
\end{proof}

We observe that in the above proof a chain $\{A_\alpha:\alpha<\kappa^+\}$ of open sets is inductively constructed by adding on members of a $\pi$-base. The closure of the union of the chain is then a regular-closed set $F$, and it follows that $wL(F,X)\leq wL(X)$. This procedure generalizes the case where $X$ has a dense set of isolated points given in Theorem 2.5(b) in~\cite{Carlson2013}.

The following corollary gives a bound for the $\theta$-density $d_\theta(X)$ of any space $X$ with a $\pi$-base consisting of ``small'' elements.
\begin{corollary}\label{cor1}
Let $X$ be a space with an open $\pi$-base $\scr{B}$ such that $|B|\leq 2^{wL(X)\chi(X)}$ for all $B\in\scr{B}$. Then $d_\theta(X)\leq 2^{wL(X)\chi(X)}$.
\end{corollary}
\begin{proof}
Let $\kappa=wL(X)\chi(X)$. We note that $\scr{C}=\{\{x\}: x\in X\}$ is a collection of compact subsets of $X$ with character at most $\kappa$. By Theorem~\ref{thm1}, there exists a subcollection $\scr{C}^\prime\sse\scr{C}$ such that $X=cl_\theta\left(\Un\scr{C}^\prime\right)$ and $|\scr{C}^\prime|\leq 2^\kappa$. Then, as $\scr{C}^\prime$ consists of singletons, we have $d_\theta(X)\leq\left|\Un\scr{C}^\prime\right|=|\scr{C}^\prime|\leq 2^\kappa$.
\end{proof}

\begin{corollary}\label{cor2}
Let $X$ be Urysohn or quasiregular, and suppose $X$ has an open $\pi$-base $\scr{B}$ such that $|B|\leq 2^{wL(X)\chi(X)}$ for all $B\in\scr{B}$. Then $|X|\leq 2^{wL(X)\chi(X)}$. 
\end{corollary}

\begin{proof}
If $X$ is Urysohn it was shown in~\cite{bel88} that $|X|\leq d_\theta(X)^{\chi(X)}$.Thus, by Corollary~\ref{cor1},
$$|X|\leq d_\theta(X)^{\chi(X)}\leq\left(2^{wL(X)\chi(X)}\right)^{\chi(X)}=2^{wL(X)\chi(X)}.$$
If $X$ is quasiregular then $d_\theta(X)=d(X)$. As $|X|\leq d(X)^{\chi(X)}$ for any Hausdorff space, we have $|X|\leq d_\theta(X)^{\chi(X)}$ and the result follows again from Cor.~\ref{cor1}.
\end{proof}

The space $Z$ in Example 2.3 in \cite{BGW1978} demonstrates that in Corollary~\ref{cor2} the condition that $X$ is Urysohn or quasiregular is necessary. We give a brief description of the space $Z$. Let $\kappa$ be an arbitrary uncountable cardinal and let $A$ be any countable dense subset of $\mathbb{P}$. Let $Z$ be the set $(\mathbb{Q}\times\kappa)\un A$. If $q\in\mathbb{Q}$ and $\alpha<\kappa$, then a neighborhood base at $(q,\alpha)$ is $\{U_n(q,\alpha):n\in\mathbb{N}\}$ where $U_n(q,\alpha)=\{(r,\alpha):r\in\mathbb{Q}\textup{ and }|r-q|<1/n\}$. If $\alpha\in A$, a neighborhood base at $a$ is $\{\{b\in A:|b-a|<1/n\}\un\{(q,\alpha):\alpha<\kappa\textup{ and }|q-a|<1/n\}:n\in\mathbb{N}\}$. $Z$ is an example of a weakly Lindel\"of, first countable Hausdorff space with arbitrary cardinality $\kappa$. Furthermore, $Z$ has a $\pi$-base of countable sets, namely the sets $U_n(q,\alpha)$ for $q\in\mathbb{Q}$ and $\alpha<\kappa$. Thus, in Corollary~\ref{cor2}, the condition that $X$ is Urysohn or quasiregular cannot be weakened to Hausdorff. 

\begin{corollary}\label{cor3}
Let $X$ be a space and $\scr{B}$ an open $\pi$-base. Suppose for all $B\in\scr{B}$ that $clB$ is H-closed, normal, Lindel\"of, or has the countable chain condition. Then $d_\theta(X)\leq 2^{wL(X)\chi(X)}$ and if $X$ is quasiregular or Urysohn then $|X|\leq 2^{wL(X)\chi(X)}$.
\end{corollary}

\begin{proof}
Let $\kappa=wL(X)\chi(X)$ and let $B\in\scr{B}$. If $clB$ is H-closed, then $|clB|\leq 2^{\psi_c(clB)}\leq 2^{\psi_c(X)}\leq 2^{\kappa}$ by Corollary 2.3 in \cite{DowPorter82}. If $clB$ is normal, then $|clB|\leq 2^{wL(clB)\chi(clB)}\leq 2^\kappa$ by Theorem 2.1 in \cite{BGW1978} and that fact that $wL(X)$ is hereditary on regular-closed sets. If $clB$ is Lindel\"of then $|clB|\leq 2^{L(clB)\chi(clB)}= 2^{\chi(clB)}\leq 2^{\chi(X)}\leq 2^\kappa$. If $clB$ has the countable chain condition, then $|clB|\leq 2^{c(clB)\chi(clB)}\leq 2^{\chi(X)}\leq 2^\kappa$. In all cases we see that $|B|\leq |clB|\leq 2^\kappa$. The results now follow from Corollaries~\ref{cor1} and~\ref{cor2}.
\end{proof}

\begin{corollary}\label{cor4}
Let $X$ be a space that is either quasiregular or Urysohn. Suppose there exists a dense subset $D\sse X$ such that each $d\in D$ has a closed neighborhood that is H-closed, normal, or Lindel\"of, or c.c.c. Then $|X|\leq 2^{wL(X)\chi(X)}$.
\end{corollary}
\begin{proof}
We show that $X$ has a $\pi$-base with the property described in Cor.~\ref{cor3}. Let $U$ be a non-empty open set in $X$. There exists $d\in U\meet D$ and an open set $V$ containing $d$ such that $clV$ is H-closed, normal, Lindel\"of, or c.c.c. Let $B=U\meet V$. Then $d\in B$. As the H-closed and c.c.c properties are hereditary on regular-closed sets, and normality and Lindel\"of-ness are closed hereditary, we see that $clB$ is H-closed, normal, Lindel\"of, or c.c.c. As $B\sse U$, this shows $X$ has a $\pi$-base with the properties described in Cor.~\ref{cor3}. Thus, $|X|\leq 2^{wL(X)\chi(X)}$.
\end{proof}

The following result of Dow and Porter~\cite{DowPorter82} follows immediately, as any space with a dense set of isolated points is quasiregular.

\begin{corollary}[Dow-Porter]\label{cor5}
If $X$ has a dense set of isolated points then $|X|\leq 2^{wL(X)\chi(X)}$.
\end{corollary}

We show in Theorem~\ref{thm1.5} below that if $X$ is regular with a $\pi$-base whose elements have compact closure, then $|X|\leq 2^{wL(X)\psi(X)t(X)}$. One can view this theorem is a variation of Corollary~\ref{cor3} above, where the hypotheses are strengthened resulting in a stronger cardinality bound. Note that the proof of Theorem~\ref{thm1.5} does not use Theorem~\ref{thm1} nor any of its corollaries.

\begin{theorem}\label{thm1.5}
If $X$ is a regular space with a $\pi$-base whose elements have compact closure, then $|X|\leq 2^{wL(X)\psi(X)t(X)}$. 
\end{theorem}

\begin{proof}
Let $\kappa=wL(X)\psi(X)t(X)$ and let $\scr{B}$ be the collection of all open sets with compact closure. As $X$ is regular, we may find for any $p\in X$ a family of open sets $\scr{U}_p$ such that $\{p\}=\Meet\scr{U}_p=\Meet\{clU:U\in\scr{U}_p\}$ and $|\scr{U}_p|\leq\kappa$. We may assume without loss of generality that each $\scr{U}_p$ is closed under finite intersections. We construct by transfinite induction a non-decreasing chain of open sets $\{A_\alpha:\alpha<\kappa^+\}$ such that (1) $cl(A_\alpha)\leq 2^{\kappa}$ for all $\alpha<\kappa^+$, and (2) if $X\minus cl(\Un\scr{W})\neq\es$ for $W\in\left[\Un\{\scr{U}_p:p\in A_\alpha\}\right]^{\leq\kappa}$, then $A_{\alpha+1}\minus cl(\Un\scr{W})\neq\es$.

For limit ordinals $\beta<\kappa^+$, let $A_\beta=\Un_{\alpha<\beta}A_{\alpha}$. Then $|A_\beta|\leq 2^\kappa$. Using that $|Y|\leq d(Y)^{\psi(Y)t(Y)}$ for any regular space $Y$, it follows that $|cl(A_\beta)|\leq 2^\kappa$. For a successor ordinal $\beta+1$, for every $\scr{W}\in\left[\Un\{\scr{U}_p:p\in A_\beta\}\right]^{\leq\kappa}$ such that $X\minus cl(\Un\scr{W})\neq\es$, we chose $B_\scr{W}\in\scr{B}$ such that
$cl(B_\scr{W})\in X\minus cl(\Un\scr{W})$. As each $cl(B_\scr{W})$ is compact, we see that $|cl(B_\scr{W})|\leq 2^{\psi(X)}\leq 2^\kappa$. We define 
$$A_{\beta+1}=A_\beta\un\Un\{B_\scr{W}: \scr{W}\in\left[\Un\{\scr{U}_p:p\in A_\beta\}\right]^{\leq\kappa}\textup{ and }X\minus cl(\Un\scr{W})\neq\es\}.$$
We again see that $|cl(A_{\beta+1)}|\leq 2^\kappa$, using that $|Y|\leq d(Y)^{\psi(Y)t(Y)}$ for any regular space $Y$.\\
Let $F=\Un\{cl(A_\alpha):\alpha<\kappa^+\}$ and note $|F|\leq 2^\kappa$. We show that $X=F$. By $t(X)\leq\kappa$, we have $F=cl\left(\Un\{A_\alpha:\alpha<\kappa^+\}\right)$ and so $F$ is regular-closed. Suppose that $X\neq F$ and choose a non-empty $B\in\scr{B}$ such that $clB\sse X\minus F$. Observe that for any $p\in F$, we have $\Meet\{clB\meet clU:U\in\scr{U}_p\}=\es$. Now, the compactness of $cl B$ and the fact that $\scr{U}_p$ is closed under finite intersections ensure the existence of some $U_p\in\scr{U}_p$ such that $U_p\meet clB=\es$. As $\{U_p:p\in F\}$ is an open cover of $F$ and $wL(X)$ is hereditary on regular-closed sets, there exists $\scr{W}\in\{U_p:p\in F\}^{\leq\kappa}$ such that $F\sse cl\left(\Un\scr{W}\right)$. There exists $\alpha<\kappa^+$ such that $\scr{W}\in\left[\Un\{\scr{U}_p:p\in A_\alpha\}\right]^{\leq\kappa}$. As $B\meet\Un\scr{W}=\es$ it follows that $X\minus cl\left(\Un\scr{W}\right)\neq\es$. Thus, $B_\scr{W}$ is defined and we have $\es\neq B_\scr{W}\sse A_{\alpha+1}\minus cl\left(\Un\scr{W}\right)\sse F\minus cl\left(\Un\scr{W}\right)=\es$. As this is a contradiction we have $X=F$ and hence $|X|\leq 2^\kappa$.
\end{proof}
The following corollary is immediate from Theorem~\ref{thm1.5}.
\begin{corollary} 
If $X$ is regular with a dense set of isolated points, then $|X|\leq 2^{wL(X)\psi(X)t(X)}$.
\end{corollary}

In Question 4.1 in~\cite{BGW1978}, it was asked whether $|X|\leq 2^{wL(X)\psi(X)t(X)}$ for a normal space $X$. Example 2.4 in that paper, described below, demonstrates that $|X|\leq 2^{wL(X)\psi(X)t(X)}$ does not hold for all regular spaces (nor, in fact, even for all zero-dimensional spaces). However, as we see in Theorem~\ref{thm1.5} above, if $X$ is regular with the added requirement that $X$ has a $\pi$-base whose elements have compact closure, then $|X|\leq 2^{wL(X)\psi(X)t(X)}$. This gives a partial answer to Question 4.1 in~\cite{BGW1978}. For the reader's benefit we give a description below of the space $Y$ in Example 2.4 in~\cite{BGW1978}.

\begin{example}[Example 2.4 in~\cite{BGW1978}]\label{example1}
Let $\kappa$ be any uncountable cardinal, let $\mathbb{Q}$ denote the rationals, and let $A$ be any countable dense subset of the space of irrational numbers. Let $Y$ be the set $(\mathbb{Q}\times\kappa)\un A$ with the following topology. If $q\in\mathbb{Q}$ and $\alpha<\kappa$ then a neighborhood base at $(q,\alpha)$ is $\{U_n(q,\alpha):n=1, 2, \ldots\}$ where $U_n(q,\alpha)=\{(r,\alpha):r\in\mathbb{Q}\textup{ and }|r-q|<1/n\}$. If $a\in A$, $n\in\mathbb{N}$, and $F\in[\kappa]^{<\omega}$, let $V_{n,F}(a)=\{b\in A:|b-a|\leq 1/n\}\un \{(q,\alpha)\in\mathbb{Q}\times\kappa:|q-\alpha|<1/n\textup{ and }\alpha\notin F\}$. Then $\{V_{n,F}(a):n\in\mathbb{N}\textup{ and }F\in[\kappa]^{<\omega}\}$ is a neighborhood base at $a$. 

It can be shown that the space $Y$ is a zero-dimensional Hausdorff space such that $wL(Y)=\aleph_0$, $t(Y)=\aleph_0$, and every subset of $Y$ is $G_\delta$ (in particular, $\psi(Y)=\aleph_0$). Thus, if $\kappa>\mathfrak{c}$, we have $|Y|=\kappa>2^{wL(Y)t(Y)\psi(Y)}$. We further observe that the family $\{U_n(q,\alpha):q\in\mathbb{Q}, \alpha<\kappa, n\in\mathbb{N}\}$ is a $\pi$-base for $Y$ consisting of countable sets. Thus, the condition in Theorem~\ref{thm1.5} above that $X$ have a $\pi$-base whose elements have compact closure cannot be changed to the condition that $X$ have a $\pi$-base whose elements are countable (or Lindel\"of or c.c.c).\qed
\end{example}

For a property $\scr{P}$ of a space $X$, we say $X$ is \emph{locally} $\scr{P}$ if every point in $X$ has a neighborhood with property $\scr{P}$. We give a short proof below that if a Hausdorff space $X$ has one of several local properties then $|X|\leq 2^{wL(X)\chi(X)}$. Note that in the case of the property locally normal we make the additional requirement that $X$ be regular. The proof of Theorem~\ref{thm2} does not use Theorem~\ref{thm1}, yet Theorem~\ref{thm2} follows from Corollary~\ref{cor4} in the case where $X$ is Urysohn or quasiregular.

\begin{theorem}\label{thm2}
Let $X$ be locally H-closed, locally Lindel\"of, locally c.c.c., or regular and locally normal.
Then $|X|\leq 2^{wL(X)\chi(X)}$.
\end{theorem}

\begin{proof}
Let $\kappa=wL(X)\chi(X)$. One can find a cover
$\mathcal{U}$ of $X$ such that $U$ has a non-empty interior and 
is H-closed, 
Lindel\"of or c.c.c for each $U\in\mathcal{U}$. Moreover, if 
$X$ is regular and locally normal,  then the cover $\mathcal  
{U}$
can consist of regular closed normal subspaces. Note that
$|U|\leq 2^\kappa$ for all $U\in\scr{U}$ (the last case uses the
fact that $wL(U)\le wL(X)$). Choose a subfamily
$\mathcal{V}\sse\mathcal{U}$ such that $|\mathcal{V}|\leq wL(X)$
and  $\cup\mathcal{V}$ is dense in $X$. To finish, observe that
$\cup\mathcal{V}$ has cardinality $\leq 2^{wL(X)\chi(X)}$  and
apply the well-known inequality $|X|\leq d(X)^{\chi(X)}$, true
for every Hausdorff space.
\end{proof}

The last instance of the previous theorem is a proper strengthening of Bell-Ginsburg-Woods's result.

\begin{corollary}\label{cor6}
If $X$ is locally compact then $|X|\leq 2^{wL(X)\psi(X)}$.
\end{corollary}
\begin{proof}
Follows from Theorem~\ref{thm2} or Theorem~\ref{thm1.5} and the fact that $X$ is of pointwise countable type; that is, $X$ can be covered by compact subsets $K$ such that $\chi(K,X)=\aleph_0$. (See 3.3.H in~\cite{Engelking}, for example). It is well-known that $\psi(X)=\chi(X)$ for every space of pointwise countable type.
\end{proof}

In \cite{arh1969} Arhangel'ski{\u\i}  derived from his general
theorem  the proof that  for any Lindel\"of Hausdorff sequential
space $X$ with $\psi(X)\le 2^{\aleph_0}$  the bound $|X|\le
2^{\aleph_0}$ holds. We will show that a similar result is true
for normal weakly Lindel\"of spaces.
\begin{theorem} If $X$ is a weakly Lindel\"of normal sequential space satisfying
$\chi(X)\le 2^{\aleph_0}$, then $|X|\le 2^{\aleph_0}$.
\end{theorem}
\begin{proof}  For each $p\in X$ let $\mathcal U_p$ be  a base of
open neighbourhoods  at $p$ satisfying $|\mathcal U_p|\le
2^{\aleph_0}$.  We will construct a non-decreasing  collection
$\{F_\alpha :\alpha <\omega_1\}$ of closed subsets of $X$ in such
a way that:

1) $|F_\alpha |\le  2^{\aleph_0}$ for each $\alpha $;

2) if $X\setminus \overline {\bigcup\mathcal V}\ne \emptyset$ 
for a countable $\mathcal V\subset \bigcup\{\mathcal U_x:x\in
F_\alpha \}$, then $F_{\alpha +1}\setminus \overline {\bigcup
\mathcal V}\ne \emptyset $.
 Fix a choice function $\phi:\mathcal P(X)\rightarrow X$ and let
$F_0=\{\phi(\emptyset )\}$.  Now, assume to have already
constructed  $\{F_\beta:\beta<\alpha \}$.  If $\alpha $ is a
limit ordinal, then put $F_\alpha =\overline
{\bigcup\{F_\beta:\beta<\alpha \}}$.  Condition 1) is fulfilled
because in a sequential space we always have $|\overline S|\le
|S|^{\aleph_0}$. If $\alpha =\gamma +1$, then  $F_\alpha $ will
be the closure of the set $F_\gamma \cup \{\phi(X\setminus
\overline {\bigcup \mathcal V}) : \mathcal V $ a countable subset
of $\bigcup \{\mathcal U_x:x\in F_\gamma\}\}$. Since a sequential
space  has countable tightness,  the set $F=\bigcup\{F_\alpha
:\alpha <\omega_1\}$ is closed and obviously we have $|F|\le
2^{\aleph_0}$. So, to finish the proof it suffices to show that
$X=F$. Assume the contrary and fix a non-open set $W$ such that
$\overline W \cap F=\emptyset $.  For every $x\in F$ choose an
element $U_x\in \mathcal U_x$ satisfying $U_x\cap \overline
W=\emptyset $. As $X$ is weakly Lindel\"of and normal, and $F$ is closed,  there is a
countable collection $\mathcal V\subset \{U_x:x\in F\}$ such that 
$F\subset \overline {\bigcup\mathcal V}$.   But, there exists
some $\gamma <\omega_1$ satisfying $\mathcal V\subset \{U_x:x\in
F_\gamma\}$ and this leads to a contradiction with condition 2).
\end{proof}

With  minor changes in the previous proof, we get a version
similar to \ref{thm1.5}.
\begin{theorem} Let $X$ be a weakly Lindel\"of, normal, sequential space satisfying
$\psi(X)\le 2^{\aleph_0}$. If $X$ has a $\pi$-base whose elements
have compact closure, then $|X|\le 2^{\aleph_0}$. \end{theorem}

Fundamental questions remain concerning which spaces $X$ have cardinality bounded by $2^{wL(X)\chi(X)}$. As locally compact spaces are Tychonoff as well as Baire, we ask the following two general questions.

\begin{question}\label{q1} 
If $X$ is Tychonoff, is $|X|\leq 2^{wL(X)\chi(X)}$?
\end{question}

\begin{question}\label{q2} 
If $X$ is a Baire space that is quasiregular or Urysohn, is $|X|\leq 2^{wL(X)\chi(X)}$?
\end{question}

Recall a space $X$ is \emph{feebly compact} if for every countable open cover $\scr{U}$ of $X$ there exists $\scr{V}\in[\scr{U}]^{<\omega}$ such that $X=\Un_{V\in\scr{V}} clV$. It is well known that a feebly compact quasiregular space is Baire (see, for example, 7T(2) in~\cite{por88}). We then ask the following question, which is a special case of Question~\ref{q2}.

\begin{question}\label{q3}
If $X$ is a feebly compact quasiregular space, is $|X|\leq 2^{wL(X)\chi(X)}$?
\end{question}

Finally, as the properties pseudocompact and feebly compact are equivalent in the class of Tychonoff spaces, we also ask the following, which is a special case of both Questions~\ref{q1} and~\ref{q3}.

\begin{question}
If $X$ is pseudocompact and Tychonoff, is $|X|\leq 2^{wL(X)\chi(X)}$?
\end{question}

\section{Power homogeneity.}
The fundamental closing-off argument of Theorem~\ref{thm1} can be used to develop a bound for the cardinality of a locally compact, power homogeneous space (Corollary~\ref{lcpowerhomog}). We begin with a bound on the $\theta$-density of a power homogeneous space with a $\pi$-base consisting of ``small'' elements. Theorem~\ref{homogthm} can be considered a companion theorem to Corollary~\ref{cor1}.

For a space $X$, recall that the \emph{pointwise compactness type of} $X$, denoted by pct($X$), is the least cardinal $\kappa$ such that $X$ can be covered by compact subsets $K$ such that $\chi(K,X)\leq\kappa$. The cardinal invariant pct($X$) then generalizes the notion of pointwise countable type.

\begin{theorem}\label{homogthm}
Let $X$ be a power homogeneous space and suppose that $X$ has a open $\pi$-base $\scr{B}$ such that $|B|\leq 2^{wL(X)t(X)\textup{pct}(X)}$ for all $B\in\scr{B}$. Then $d_\theta(X)\leq 2^{wL(X)t(X)\textup{pct}(X)}$.
\end{theorem}

\begin{proof}
The proof is a modification of the proof of Corollary 3.3(b) in~\cite{Carlson2013}. Let $\kappa=wL(X)t(X)\textup{pct}(X)$. As $t(X)\textup{pct}(X)\leq\kappa$, by Lemma 3.1 in~\cite{Carlson2013} there exists a non-empty compact set $K$ and a set $A\in[X]^{\leq\kappa}$ such that $\chi(K,X)\leq\kappa$ and $K\sse clA$. As $\pi_\chi(X)\leq t(X)\textup{pct}(X)\leq\kappa$, a well-known generalization of \sapirovskii's inequality $\pichar{X}\leq t(X)$ for a compactum $X$, by Corollary 2.9 in~\cite{avr2007} there exists a cover $\scr{C}$ of $X$ of compact sets of character at most $\kappa$ such that each member of $\scr{C}$ is contained in the closure of a subset of cardinality at most $\kappa$. By Theorem~\ref{thm1}, there exists a subcollection $\scr{C}^\prime\sse\scr{C}$ such that $X=cl_\theta(\Un\scr{C}^\prime)$ and $|\scr{C}^\prime|\leq 2^\kappa$.

Now, for each $C\in\scr{C}^\prime$ there exists $A_C\in[X]^{\leq\kappa}$ such that $C\sse cl(A_C)$. Let $\scr{A}=\{A_C:C\in\scr{C}^\prime\}$ and note $|\scr{A}|\leq 2^\kappa$. For each $C\in\scr{C}$ we have that $C\sse cl(\Un\scr{A})$ and therefore $\Un\scr{C}^\prime\sse cl\left(\Un\scr{A}\right)\sse cl_\theta\left(\Un\scr{A}\right)$. It follows that $X=cl_\theta\left(\Un\scr{C}^\prime\right)\sse cl_\theta\left(\Un\scr{A}\right)$, and so $d_\theta(X)\leq|\Un\scr{A}|\leq 2^\kappa\cdot\kappa=2^\kappa$.
\end{proof}

In Theorem 4.3 in~\cite{car07} it was shown that $|X|\leq d_\theta(X)^{\pichar{X}}$ for a Urysohn power homogeneous space $X$. As $d_\theta(X)=d(X)$ for a quasiregular space $X$ and $|X|\leq d(X)^{\pichar{X}}$ for any (Hausdorff) power homogeneous space $X$ by a result of Ridderbos, Theorem~\ref{homogthm} yields the following corollary.

\begin{corollary}\label{homogcor1}
Let $X$ be a power homogeneous Hausdorff space that is either quasiregular or Urysohn. Suppose that $X$ has a open $\pi$-base $\scr{B}$ such that $|B|\leq 2^{wL(X)t(X)\textup{pct}(X)}$ for all $B\in\scr{B}$. Then $|X|\leq 2^{wL(X)t(X)\textup{pct}(X)}$.
\end{corollary}

\begin{proof}
Let $\kappa=wL(X)t(X)\textup{pct}(X)$. If $X$ is quasiregular then $d_\theta(X)=d(X)$ and, by Theorem 3.4 in~\cite{rid2006}, $|X|\leq d(X)^{\pichar{X}}=d_\theta(X)^{\pichar{X}}$. If $X$ is Urysohn, then $|X|\leq d_\theta(X)^{\pichar{X}}$ by Theorem 4.3~\cite{car07}. In either case, by Theorem~\ref{homogthm} above, we have $|X|\leq d_\theta(X)^{\pichar{X}}\leq (2^\kappa)^{\pichar{X}}\leq 2^\kappa$.
\end{proof}

We move towards establishing a cardinality bound on any open subset of a power homogeneous space $X$ (Theorem~\ref{phtheorem}). That theorem can be regarded as an improvement of Corollary 3.11 in~\cite{CR2012}. We first obtain some preliminary lemmas. Lemmas~\ref{phlemma1} and~\ref{phlemma2} are essentially slight improvements of Corollary 3.3 and Theorem 3.4 from~\cite{rid2006} and we use notation and proof techniques similar to that paper. For a space $X$, $x\in X$, a cardinal $\mu$, and product space $X^\mu$, we denote by $\bar{x}$ the element of $X^\mu$ which is equal to $x$ on all coordinates. The \emph{diagonal} of $X^\mu$ is $\Delta(X,\mu)=\{\bar{x}:x\in X\}$. For a subset $C\sse\mu$, let $\pi_C:X^\mu\to X^C$ be the projection, and for $x\in X$ let $x_C$ be the point $\pi_C(\bar{x})\in X^C$. We also adopt the notation given before and after Theorem 3.1 in~\cite{rid2006}, which features heavily in the next proof.

\begin{lemma}\label{phlemma1}
Let $X$ be a space such that $X^\mu$ is homogeneous for a cardinal $\mu$. Fix $p\in X$ and suppose $\pichar{X}\leq\kappa$ for a cardinal $\kappa$. Let $\scr{U}$ be a local $\pi$-base at $p$ in $X$ such that $|\scr{U}|\leq\kappa$, and suppose $\mu\geq\kappa$. Suppose further that $D\sse X$ and $Y\sse clD$. Let $\pi: X^\mu\to X$ be any fixed projection. Then for all $x\in X$, there exists a homeomorphism $h_x:X^\mu\to X^\mu$ such that
\begin{itemize}
\item [(1)] $h_x(\bar{p})=\bar{x}$, and
\item [(2)] if $B\in\scr{U}(\kappa)$ satisfies $\pi h_x\pi_\kappa^{-1}[B]\meet Y\neq\es$, then there exists $e\in X^\mu$ and $d\in D$ satisfing
	\begin{itemize}
	\item [(a)] $\pi h_x(e)=d\in D$ and $e\in\pi_\kappa^{-1}[B]$, and
	\item [(b)] $h_x\pi_\kappa^{-1}(e_\kappa)\sse \pi^{-1}(d)$.
	\end{itemize}
\end{itemize}
\end{lemma}

\begin{proof}
Fix $x\in X$. As $X^\mu$ is homogeneous, there exists a homeomorphism $h:X^\mu\to X^\mu$ such that $h(\bar{p})=\bar{x}$. We inductively construct a sequence $\{A_n:n<\omega\}\sse [\mu]^{\leq\kappa}$ of subsets such that $|A_n|\leq\kappa$ for all $n<\omega$. Let $A_0=\kappa$ and suppose $A_n$ has been defined. 

Let $C\in\scr{U}(A_n)$ and suppose that $\pi h\pi^{-1}_{A_n}[C]\meet Y\neq\es$. Then $\pi h\pi^{-1}_{A_n}[C]\meet clD\neq\es$ and thus $\pi h\pi^{-1}_{A_n}[C]\meet D\neq\es$. Let $d(C)\in\pi h\pi^{-1}_{A_n}[C]\meet D$, and select $e(C)\in\pi_{A_n}^{-1}[C]$ such that $d(C)=\pi h(e(C))$. If $\pi h\pi^{-1}_{A_n}[C]\meet Y=\es$, select any $e(C)\in\pi^{-1}_{A_n}[C]$. Let $Z=\{e(C):C\in\scr{U}(A_n)\}$ and note that $|Z|\leq|\scr{U}(A_n)|\leq\kappa$. Now apply Cor. 2.3 in~\cite{avr2007} (also Prop. 2.2 in~\cite{rid2006}) to obtain a set $A_{n+1}$ such that $|A_{n+1}|\leq\kappa$ and $h\pi^{-1}_{A_{n+1}}((e(C))_{A_{n+1}})\sse\pi^{-1}\pi h(e(C))$ for all $C\in\scr{U}(A_n)$. Note that if $C\in\scr{U}(A_n)$ is such that $\pi h\pi^{-1}_{A_n}[C]\meet Y\neq\es$, then $h\pi^{-1}_{A_{n+1}}((e(C))_{A_{n+1}})\sse\pi^{-1}(d(C))$. 

Let $A=\Un_{n<\omega}A_n$. As $\kappa=A_0\sse A$ and $|A|\leq\kappa$, we see that $|A|=\kappa$. There are two cases. First, if $\mu=\kappa$, then $A=\kappa$ and we have then proved above that $h$ is the required homeomorphism satisfying the properties needed in the statement of the Lemma. 

Otherwise, for the rest of the proof we suppose $\kappa <\mu$. In this case we use the coordinate change homeomorphism $g_{\kappa\to A}:X^\mu\to X^\mu$ (see notation following Thm 3.1 in~\cite{rid2006}). Let $g= g_{\kappa\to A}$ and define $h_x=h\circ g$. Note $h_x(\bar{p})=h(g(\bar{p}))=h(\bar{p})=\bar{x}$, establishing (1) above.

We prove (2) holds for the homeomorphism $h_x$. Let $B\in\scr{U}(\kappa)$ and suppose that $\pi h_x\pi_\kappa^{-1}[B]\meet Y\neq\es$. We see that $(g[\pi^{-1}_\kappa[B]])_A\in\scr{U}(A)$ by Lemma 3.2 in~\cite{rid2006}. Let $B^\prime=(g[\pi^{-1}_\kappa[B]])_A$. Also by Lemma 3.2 in~\cite{rid2006}, we have that 

\begin{align}
\es\neq \pi h_x\pi_\kappa^{-1}[B]\meet Y&=\pi h[g[\pi_\kappa[B]]]\meet Y\notag\\
&=\pi h\pi_A^{-1}[(g[\pi^{-1}_\kappa[B]])_A]\meet Y\notag\\
&=\pi h\pi_A^{-1}[B^\prime]\meet Y.\notag
\end{align}

Now, as noted at the end of the proof of Thm 3.1 in~\cite{rid2006}, $\scr{U}(A)=\Un_{n<\omega}\pi^{-1}_{A\to A_{n+1}}\scr{U}(A_{n+1})$. Thus there exists $n<\omega$ and $C\in\scr{U}(A_n)$ such that $B^\prime=\pi^{-1}_{A\to A_n}[C]$. Now, for that $n$,
\begin{align}
\pi h\pi^{-1}_{A_n}[C]\meet Y&=\pi h\pi^{-1}_A[\pi^{-1}_{A\to A_n}[C]]\meet Y\notag\\
&=\pi h\pi_A^{-1}[B^\prime]\meet Y\notag\\
&\neq\es.\notag
\end{align}
As $C\in\scr{U}(A_n)$ and $\pi h\pi^{-1}_{A_n}[C]\meet Y\neq\es$, we have $d(C)$ and $e(C)$ as defined in the second paragraph of this proof. Let $d=d(C)$ and $e^\prime=e(C)$. Then $d\in\pi h\pi^{-1}_{A_n}[C]\meet D$, $e^\prime\in\pi_{A_n}^{-1}[C]$,  $d=\pi h(e^\prime)$, and $h\pi^{-1}_{A_{n+1}}(e^\prime_{A_{n+1}})\sse\pi^{-1}(d)$. But $\pi_A^{-1}(e^\prime_A)\sse\pi^{-1}_{A_{n+1}}(e^\prime_{A_{n+1}})$ and thus $h\pi^{-1}_A(e^\prime_A)\sse h\pi^{-1}_{A_{n+1}}(e^\prime_{A_{n+1}})\sse\pi^{-1}(d)$. Furthermore, by Lemma 3.2 in~\cite{rid2006},
\begin{align}
d\in\pi h\pi^{-1}_{A_n}[C]\meet D&=\pi h\pi^{-1}_A[\pi^{-1}_{A\to A_n}[C]]\meet D=\pi h\pi^{-1}_A[B^\prime]\meet D\notag\\
&=\pi h\pi^{-1}_A[(g[\pi^{-1}_\kappa[B]])_A]=\pi hg[\pi^{-1}_\kappa[B]]\notag\\
&=\pi h_x\pi^{-1}_\kappa[B],\notag
\end{align} 
and 
\begin{align} 
e^\prime\in\pi_{A_n}^{-1}[C]&=\pi^{-1}_A[\pi^{-1}_{A\to A_n}[C]]=\pi^{-1}_A[B^\prime]\notag\\
&=\pi^{-1}_A[(g[\pi^{-1}_\kappa[B]])_A]=g[\pi_\kappa^{-1}[B]].\notag
\end{align}

Thus there exists $e\in\pi_\kappa^{-1}[B]$ such that $e^\prime=g(e)$. Thus, $d=\pi h(e^\prime)=\pi hg(e)=\pi h_x(e)$. Furthermore, as $\pi_A^{-1}(e^\prime_A)=g\pi_\kappa^{-1}(e_\kappa)$ (again by Lemma 3.2 in~\cite{rid2006}), we have that $h_x\pi^{-1}_\kappa(e_\kappa)=hg\pi_\kappa^{-1}(e_\kappa)=h\pi_A^{-1}(e^\prime_A)\sse\pi^{-1}(d)$. This establishes the Lemma.
\end{proof}

The next lemma is a strengthening of Theorem 3.4 in~\cite{rid2006}.

\begin{lemma}\label{phlemma2}
Let $X$ be a power homogeneous space. If $D\sse X$ and $U$ is an open set such that $U\sse clD$, then $|U|\leq |D|^{\pichar{X}}$.
\end{lemma}
\begin{proof} 
Fix $p\in X$, let $\kappa=\pichar{X}$, and let $\mu$ be a cardinal such that $X^\mu$ is homogeneous. If $\mu\leq\kappa$, then $X^\kappa$ is also homogeneous, so we can assume without loss of generality that $\mu\geq\kappa$. For all $\bar{x}\in\Delta(X,\mu)$ there exists a homeomorphism $h_x:X^\mu\to X^\mu$ satisfying (1) and (2) in Lemma~\ref{phlemma1}, where $Y=U$. Now $\scr{U}(\kappa)$ is a local $\pi$-base at $p_\kappa$ in $X^\kappa$ of size at most $\kappa$. Fix $q\in D$. 

We define $\phi:\Delta(U,\mu)\to D^{\scr{U}(\kappa)}$ as follows. Let $\bar{x}\in\Delta(U,\mu)$ and $B\in\scr{U}(\kappa)$. If $\pi h_x\pi^{-1}_\kappa[B]\meet U\neq\es$, set $\phi(x)(B)=\pi h_x(e)=d\in D$, where
$e$ and $d$ are as in 2(a) and 2(b) in Lemma~\ref{phlemma1}. Otherwise, define $\phi(x)(B)=q$. Thus $\phi$ is well-defined.

We show $\phi$ is one-to-one. Suppose $\bar{x}\neq\bar{y}\in\Delta(U,\mu)$. There exist disjoint open sets $V$ and $W$ in $X$ such that $x\in V$ and $y\in W$. So $x\in V\meet U$, $y\in W\meet U$, and
$$p_\kappa\in\pi_\kappa h_x^{-1}\pi^{-1}[V\meet U]\meet\pi_\kappa h_y^{-1}\pi^{-1}[W\meet U].$$
Note that the set on the right above is open in $X^\kappa$. There exists $B\in\scr{U}(\kappa)$ such that
$$B\sse\pi_\kappa h_x^{-1}\pi^{-1}[V\meet U]\meet\pi_\kappa h_y^{-1}\pi^{-1}[W\meet U],$$
from which it follows that $\pi h_x\pi_\kappa^{-1}[B]\meet U\neq\es$ and $\pi h_y\pi_\kappa^{-1}[B]\meet U\neq\es$. 

Thus $\phi(x)(B)=\pi h_x(e)=d$ for $e,d$ satisfying conditions 2(a) and 2(b) in Lemma~\ref{phlemma1}.  So $e\in\pi^{-1}_\kappa[B]$ and $e_\kappa\in B$. Since $B\sse\pi_\kappa h_x^{-1}\pi^{-1}[V\meet U]$, it follows that $\pi_\kappa^{-1}(e_\kappa)\meet h_x^{-1}\pi^{-1}[V\meet U]\neq\es$. Applying $h_x$ we have $h_x\pi_\kappa^{-1}(e_\kappa)\meet\pi^{-1}[V\meet U]\neq\es$. Since $h_x\pi_\kappa^{-1}(e_\kappa)\sse\pi^{-1}(d)$, it follows that $d\in V\meet U$ and $\phi(x)(B)\in V\meet U$. Likewise, a similar argument shows $\phi(x)(B)\in W\meet U$.

As $V\meet W\neq\es$, it follows that $\phi(x)(B)\neq\phi(y)(B)$ and therefore $\phi(x)\neq\phi(y)$. This shows $\pi$ is an injection. It follows that $|U|=|\Delta(U,\mu)|\leq |D|^{|\scr{U}(\kappa)|}\leq |D|^\kappa$.
\end{proof}

The following theorem represents a strengthening of Corollary 3.11 in~\cite{CR2012}.
\begin{theorem}\label{phtheorem}
If $X$ be a power homogeneous space and $U\sse X$ is a non-empty open set, then $|U|\leq 2^{L(\overline{U})t(X)\textup{pct}(X)}$.
\end{theorem}

\begin{proof}
Let $\kappa=L(\overline{U})t(X)\textup{pct}(X)$ and let $K=\overline{U}$. As $t(X)\textup{pct}(X)\leq\kappa$, by Lemma 3.8 in~\cite{CPR2012} there exists a compact set of character at most $\kappa$ contained in the closure of a set of size $\kappa$. Proposition 3.2 in \cite{Carlson2013}, which is a slight variation of Corollary 2.9 in \cite{avr2007}, guarantees there exists a family $\scr{G}$ of compact sets of character at most $\kappa$ and a family of subsets $\scr{H}=\{H_G:G\in\scr{G}\}\sse[X]^{\leq\kappa}$ such that a) $G\in cl(H_G)$ for all $G\in\scr{G}$, and b) $X=\Un\scr{G}$. Note that each $G\in\scr{G}$ is a $G_\kappa$-set in $X$. (Recall a $G_\kappa$-set is an intersection of $\kappa$-many open sets.)

Now, $K=\Un\{G\meet K:G\in\scr{G}\}$ and for all $G\in\scr{G}$, if $G\meet K\neq\es$ then $G\meet K$ is a $G_\kappa$-set in the subspace $K$. Therefore, as $K$ is compact, by Theorem 4 in Pytkeev~\cite{pyt85}, there exists $\scr{G}^\prime\sse\scr{G}$ such that $|\scr{G}^\prime|\leq 2^{t(K)\cdot\kappa}\leq 2^{t(X)\cdot\kappa}=2^\kappa$ and $K\sse\Un\scr{G}^\prime$. 

Let $D=\Un\{H_G:G\in\scr{G}^\prime\}$. Note that $|D|\leq |\scr{G}^\prime|\cdot\kappa\leq 2^\kappa\cdot\kappa=2^\kappa$. We have, 
$$U\sse K\sse\Un\scr{G}^\prime\sse\Un\{cl(H_G):G\in\scr{G}^\prime\}\sse cl\left(\Un\{H_G:G\in\scr{G}^\prime\}\right)=cl D.$$
By Lemma~\ref{phlemma2} and the fact that $\pichar{X}\leq t(X)\textup{pct}(X)$ for any space $X$, it follows that $|U|\leq |D|^{\pichar{X}}\leq (2^\kappa)^{\pichar{X}}\leq (2^\kappa)^\kappa=2^\kappa$.
\end{proof}

We now tie our results in with Corollary~\ref{homogcor1}.

\begin{corollary}\label{phcor1}
Let $X$ be a power homogeneous space that is either quasiregular or Urysohn. If $X$ has a $\pi$-base $\scr{B}$ such that $clB$ is Lindel\"of for all $B\in\scr{B}$ then $|X|\leq 2^{wL(X)t(X)\textup{pct}(X)}$.
\end{corollary}

\begin{proof}
Let $\kappa=wL(X)t(X)\textup{pct}(X)$. By Theorem~\ref{phtheorem} it follows that $|B|\leq 2^{t(X)\textup{pct}(X)}\leq 2^\kappa$ for all $B\in\scr{B}$. By Corollary~\ref{homogcor1}, $|X|\leq 2^\kappa$.
\end{proof}

\begin{corollary}\label{lcpowerhomog}
Let $X$ be a locally compact power homogeneous space. Then $|X|\leq 2^{wL(X)t(X)}$. 
\end{corollary}

\begin{proof}
Note that $X$ has a $\pi$-base (in fact, a base) $\scr{B}$ of open sets such that $clB$ is compact for all $B\in\scr{B}$. Also $X$ is of pointwise countable type, i.e. $\textup{pct}(X)=\aleph_0$. By Corollary~\ref{phcor1}, $|X|\leq 2^{wL(X)t(X)\textup{pct}(X)}=2^{wL(X)t(X)}$.
\end{proof}

We conclude with a question analogous to Question~\ref{q1}.

\begin{question}
If $X$ is power homogeneous and Tychonoff, is $|X|\leq 2^{wL(X)t(X)\textup{pct}(X)}$?
\end{question}

\end{document}